\documentclass[leqno]{amsart}
\usepackage{amsmath}
\usepackage{amssymb}
\usepackage{amsthm}
\usepackage{enumerate}
\usepackage[mathscr]{eucal}
\usepackage{enumitem}
\theoremstyle{plain}
\newtheorem{theorem}{Theorem}[section]
\newtheorem{lemma}[theorem]{Lemma}
\newtheorem{prop}[theorem]{Proposition}
\theoremstyle{definition}

\newtheorem{remark}[theorem]{Remark}

\newtheorem{example}[theorem]{Example}

\newtheorem{cor}[theorem]{Corollary}
\theoremstyle{remark}
\usepackage{amsmath, amsthm, amscd, amsfonts, amssymb, graphicx, color}
\usepackage[bookmarksnumbered, colorlinks, plainpages]{hyperref}
\hypersetup{colorlinks=true,linkcolor=red, anchorcolor=green, citecolor=cyan, urlcolor=red, filecolor=magenta, pdftoolbar=true}
\begin{document}

\title [Bounds for the Davis-Wielandt radius of bounded linear operators]{Bounds for the Davis-Wielandt radius of bounded linear operators} 

\author[Pintu Bhunia, Aniket Bhanja, Santanu Bag and Kallol Paul]{ Pintu Bhunia, Aniket Bhanja, Santanu Bag and Kallol Paul}

\address [Bhunia] {Department of Mathematics, Jadavpur University, Kolkata 700032, West Bengal, India}
\email{pintubhunia5206@gmail.com}

\address[Bhanja] {Department of Mathematics\\ Vivekananda College Thakurpukur\\ Kolkata\\ West Bengal\\India\\ }
\email{aniketbhanja219@gmail.com}

\address[Bag] {Department of Mathematics\\ Vivekananda  College for Women, Barisha \\ Kolkata  \\ West Bengal\\ INDIA}
\email{santanumath84@gmail.com}

\address[Paul] {Department of Mathematics, Jadavpur University, Kolkata 700032, West Bengal, India}
\email{kalloldada@gmail.com}

%\thanks will become a 1st page footnote.
\thanks{Mr. Pintu Bhunia would like to thank UGC, Govt. of India for the financial support in the form of SRF. Prof. Kallol Paul would like to thank RUSA 2.0, Jadavpur University for  partial support. }
%\thanks{}
%\thanks{}

%    General info


%\date{}
\maketitle
\begin{abstract}
 We obtain  upper and lower bounds for the Davis-Wielandt radius of bounded linear  operators defined on a complex Hilbert space, which improve on the existing ones. We also obtain  bounds for the Davis-Wielandt radius of operator matrices. We determine the exact value of the Davis-Wielandt radius of two special type of operator matrices $\left(\begin{array}{cc}
 I & B\\
 0 & 0
 \end{array}\right)$ and $\left(\begin{array}{cc}
 0 & A\\
 B & 0
 \end{array}\right)$, where $A,B\in \mathcal{B}(\mathcal{H})$, $I$ and $0$ are the identity operator and the zero operator on $\mathcal{H},$ respectively. Finally we obtain bounds for the Davis-Wielandt radius of operator matrices of the form   $\left(\begin{array}{cc}
 A& B\\
 0 & C
 \end{array}\right),$ where $A,B, C\in \mathcal{B}(\mathcal{H}).$ 

\end{abstract}

\section{Introduction}

\smallskip

\noindent Let $\mathcal{H}$ be a complex Hilbert space with usual inner product $\langle.,.\rangle$ and $\mathcal{B}(\mathcal{H})$ be the $\mathbb{C}^*$-algebra of all bounded linear operators on $\mathcal{H}$. For $T\in \mathcal{B}(\mathcal{H})$, $T^*$ and $|T|$ denote the adjoint of $T$ and absolute value of $T$ (i.e., $|T|=(T^*T)^{\frac{1}{2}}$), respectively. The operator norm and the minimum norm of $T\in \mathcal{B}(\mathcal{H}),$ denoted by $\|T\|$ and $m(T)$, respectively, are defined as
$\|T\|=\sup\left\{\|Tx\|: x\in \mathcal{H}, \|x\|=1 \right \}$ and  $m(T)=\inf \left\{\|Tx\|: x\in \mathcal{H}, \|x\|=1 \right \}.$  For given $T\in \mathcal{B}(\mathcal{H})$, the numerical range of $T$, denoted as $W(T),$ is defined as the collection of all scalars $\langle Tx,x \rangle $ with $  \|x\|=1$, i.e., $ W(T)=\left\{\langle Tx,x\rangle : x\in \mathcal{H}, \|x\|=1 \right \}.$ The numerical radius and the Crawford number of $T,$ denoted as $w(T)$ and $c(T)$, respectively,  are defined as
\begin{eqnarray*}
	w(T)&=&\sup \left\{|\langle Tx,x\rangle| : x\in \mathcal{H}, \|x\|=1 \right \},\\
	c(T)&=&\inf \left\{|\langle Tx,x\rangle| : x\in \mathcal{H}, \|x\|=1 \right \}.
\end{eqnarray*}
The operator norm attainment set of $T,$ denoted as $M_T$, is defined as the set of all unit vectors in $\mathcal{H}$ at which $T$ attains its norm, i.e., $$ M_T =\left \{ x\in \mathcal{H} :  \|x\|=1, \|Tx\|=\|T\|   \right \}.$$ Likewise, 
the numerical radius attainment set and the Crawford number attainment set of $T,$  denoted as $W_T$ and $c_T$, respectively, are defined as
\begin{eqnarray*}
	W_T &=& \left \{ x\in \mathcal{H} :  \|x\|=1, |\langle Tx,x\rangle|=w(T)   \right \},\\
	c_T &=& \left \{ x\in \mathcal{H} :  \|x\|=1, |\langle Tx,x\rangle|=c(T)   \right \}. 
\end{eqnarray*}
It is well-known that the numerical radius $w(.)$ defines a norm on $\mathcal{B}(\mathcal{H})$, equivalent to the operator norm $\|.\|$ satisfying the following inequality. For $T\in \mathcal{B}(\mathcal{H}),$ 
$$\frac{1}{2}\|T\|\leq  w(T)\leq \|T\|. $$ 
The first inequality becomes equality if $T$ is nilpotent of index $2$ and second inequality becomes equality if $T$ is normal. The power inequality for the numerical radius, $w(T^{n})\leq w^n(T),~~~\forall~~ n=1,2,\ldots,$ is an important inequality in the study of numerical radius inequalities. Due to the importance of the numerical range and the numerical radius inequalities, these areas have attracted  many mathematicians over the years. We refer a few of the articles and books \cite{BBP,DB,GR,PB}  and the references therein for further readings.  The Davis-Wielandt radius of an operator is an important generalization of the numerical radius. The Davis-Wielandt shell $DW(T) $ and the Davis-Wielandt radius $dw(T)$ of $T\in \mathcal{B}(\mathcal{H})$ are defined as follows (see \cite{D,W}):
\begin{eqnarray*}
	DW(T)&=& \left \{ \left (\langle Tx,x\rangle, \|Tx\|^2 \right ) :  x\in \mathcal{H}, \|x\|=1   \right \} \subseteq \mathbb{C} \times \mathbb{R},\\
	dw(T)&=& \sup \left \{ \sqrt{|\langle Tx,x\rangle|^2+ \|Tx\|^4 } :  x\in \mathcal{H}, \|x\|=1   \right \}.
\end{eqnarray*}
It is easy to verify that the Davis-Wielandt radius $dw(.)$ cannot define a norm on $\mathcal{B}(\mathcal{H})$, although it satisfies the following inequality. For $T\in \mathcal{B}(\mathcal{H}),$ 
\begin{eqnarray}\label{1stbound}
\max \{ w(T), \|T\|^2 \} \leq dw(T)\leq \sqrt{w^2(T)+\|T\|^4}.
\end{eqnarray}
The inequalities in (\ref{1stbound}) are sharp, if we consider $T=\left(\begin{array}{cc}
1 & 0\\
0 & 1
\end{array}\right)$ then we have $dw(T)=\sqrt{w^2(T)+\|T\|^4} $ $=\sqrt{2}$ and if we consider $S=\left(\begin{array}{cc}
0 & 1\\
0 & 0
\end{array}\right)$ then we have $dw(S)=\max \{ w(S), \|S\|^2 \}=1.$ The second inequality in  (\ref{1stbound}) becomes equality if and only if $T$ is normaloid, i.e., $w(T)=\|T\|$ (see \cite[Cor. 3.2]{ZMCN}).  The Davis-Wielandt shell have been studied by many mathematicians that includes but not limited to Li and Poon \cite{LP}, Li et. al. \cite{LPS}, Lins et. al. \cite{LSZ}.   Recently, Zamani and Shebrawi \cite{ZS} and Zamani et. al. \cite{ZMCN} have also studied  the Davis-Wielandt radius of bounded linear operators.\\

\noindent In this paper, we begin with the study of equality of the lower bounds for Davis-Wielandt radius mentioned  in (\ref{1stbound}). We then obtain new upper and lower bounds for the Davis-Wielandt radius and show that the bounds obtained here improve on the existing ones. Further, we obtain an upper bound for the Davis-Wielandt radius of sum of two bounded linear operators, i.e., namely, 
\[dw(S+T)\leq dw(S)+dw(T)+w(S^*T+T^*S),\]
for $ S,T \in \mathcal{B}(\mathcal{H})$.
We also obtain upper and lower bounds for  the Davis-Wielandt radius of some operator matrices. The bounds for operator matrices can be used to obtain bounds for the Davis-Wielandt radius of some  bounded linear operators.  We  give numerical examples to show that bounds obtained for operator matrices are better than those obtained for bounded linear operators. We determine the exact value of the Davis-Wielandt radius of two special type of operator matrices $\left(\begin{array}{cc}
I & B\\
0 & 0
\end{array}\right)$ and $\left(\begin{array}{cc}
0 & A\\
B & 0
\end{array}\right)$, where $A,B\in \mathcal{B}(\mathcal{H})$, $I$ and $0$ are the identity operator and the zero operator on $\mathcal{H},$ respectively. Finally we obtain bounds for the Davis-Wielandt radius of operator matrices of the form   $\left(\begin{array}{cc}
A& B\\
0 & C
\end{array}\right),$ where $A,B, C\in \mathcal{B}(\mathcal{H}).$

\section{\textbf{Davis-Wielandt radius inequalities of operators}}

\smallskip
\noindent We begin this section with the following results on the equality of the lower bound for the Davis-Wielandt radius of bounded linear operators mentioned in (\ref{1stbound}). We skip the elementary details of the proofs.

\begin{prop}
Let $T\in \mathcal{B}({\mathcal{H}}).$ Then the following conditions are equivalent.\\
$(i)$ $dw(T)=w(T)$.\\
$(ii)$ $T=0$.
\end{prop}

\begin{prop}\label{th-equality2}
Let $T\in \mathcal{B}({\mathcal{H}})$ and let $dw(T)=\|T\|^2.$ Then one of the following conditions holds.\\
$(i)$ Let $M_T\neq \emptyset $. Then $|\langle Tx,x\rangle|=0$ if $x\in M_T $, i.e., $M_T \subseteq c_T.$\\
$(ii)$ Let $M_T= \emptyset $. Then there exists a sequence $\{x_n\}$ in $\mathcal{H}$ with $\|x_n\|=1$ such that $\lim_{n\to \infty}\|Tx_n\|=\|T\|$ and $\lim_{n\to \infty} |\langle Tx_n,x_n\rangle|=0.$ 
\end{prop}

\begin{remark}
We note  that the converse part of Proposition \ref{th-equality2} may not be true. As for example, if we consider $T=\left(\begin{array}{ccc}
	\frac{3}{8}&0&0\\
	0 & 0&\frac{1}{2} \\
	0 & 0&0
	\end{array}\right)$ then we see that $|\langle Tx,x\rangle|=0$ for all $x\in M_T$, i.e., $ M_T\subseteq c_T$. But, $dw(T)\neq \|T\|^2=\frac{1}{4}$ as $dw(T)\geq w(T)=\frac{3}{8}.$
\end{remark}

\smallskip

 In our first theorem of this section, we obtain new lower bounds for the Davis-Wielandt radius of bounded linear operators.

\begin{theorem}\label{th-lower1}
Let $T\in \mathcal{B}({\mathcal{H}}).$ Then the following inequalities are hold.
\begin{eqnarray*}
(i)~ dw^2(T) & \geq & \max \left\{ w^2(T)+c^2(T^*T), \|T\|^4+c^2(T)\right\},\\
(ii)~dw^2(T) & \geq & 2\max \left\{w(T)c(T^*T), c(T)\|T\|^2 \right\}.
\end{eqnarray*}
\end{theorem}

\begin{proof}
$ (i) $ Let $x$ be a unit vector in $\mathcal{H}$. Then from the definition of $dw(T)$, we get
\begin{eqnarray*}
dw^2(T) &\geq & |\langle Tx,x \rangle |^2+\|Tx\|^4\\
 &= & |\langle Tx,x \rangle |^2+\langle T^*Tx,x\rangle ^2 \\
&\geq & |\langle Tx,x \rangle |^2 + c^2(T^*T).
\end{eqnarray*}
Therefore, taking supremum over all unit vectors in $\mathcal{H}$, we have
\[dw^2(T)\geq  w^2(T)+c^2(T^*T).\]

Again from $dw^2(T) \geq  |\langle Tx,x \rangle |^2+\|Tx\|^4,$ where $\|x\|=1$, we get
\[dw^2(T) \geq  c^2(T)+\|Tx\|^4.\]
Taking supremum over all unit vectors in $\mathcal{H}$, we have
\[dw^2(T) \geq  c^2(T)+\|T\|^4.\]
This completes the proof of $(i).$\\
$(ii)$  	For all  $x\in \mathcal{H}$ with $\|x\|=1,$ we have  $$|\langle Tx,x \rangle |^2+\|Tx\|^4\geq 2|\langle Tx,x \rangle | \|Tx\|^2$$ and   so 
$$ dw^2(T)\geq 2|\langle Tx,x \rangle | \langle T^*Tx,x \rangle \geq 2|\langle Tx,x \rangle | c(T^*T).$$

Taking supremum over all unit vectors in $\mathcal{H}$, we get 
\[dw^2(T)\geq 2w(T)c(T^*T).\]

Again from $|\langle Tx,x \rangle |^2+\|Tx\|^4\geq 2|\langle Tx,x \rangle | \|Tx\|^2$, we have
\[dw^2(T)\geq 2c(T) \|Tx\|^2.\]
Taking supremum over all unit vectors in $\mathcal{H}$, we get 
\[dw^2(T)\geq 2c(T)\|T\|^2.\]
This completes the proof.
\end{proof}

\begin{remark}\label{rem-1}
 1. It is clear that the inequality obtained in Theorem \ref{th-lower1} $(i) $ improves on  the first inequality in (\ref{1stbound}). \\
2. If $c(T) > \frac{\|T\|^2}{2}$ and $w(T)\leq \|T\|^2,$ then it is easy to see that the inequality in Theorem \ref{th-lower1} $(ii) $  is sharper than the first inequality in (\ref{1stbound}). 
\end{remark}

 In the following theorem, we obtain an upper bound for the Davis-Wielandt radius of bounded linear operators.

\begin{theorem}\label{th-upper2}
Let $T\in \mathcal{B}({\mathcal{H}}).$ Then 
\[dw^2(T)\leq \sup_{\theta \in \mathbb{R}}w^2(e^{{\rm i} \theta}T+T^*T)-2c(T)m^2(T).\]
\end{theorem}

\begin{proof}
Let $x\in \mathcal{H}$ with $\|x\|=1.$ Then there exists $\theta \in \mathbb{R}$ such that $|\langle Tx,x \rangle |=e^{{\rm i} \theta}\langle Tx,x\rangle$. Now,
\begin{eqnarray*}
|\langle Tx,x \rangle |^2+\|Tx\|^4 &=& \langle e^{{\rm i} \theta}Tx,x\rangle^2+ \langle T^*Tx,x\rangle ^2 \\
&=& (\langle e^{{\rm i} \theta}Tx,x\rangle+ \langle T^*Tx,x\rangle )^2-2\langle e^{{\rm i} \theta}Tx,x\rangle \langle T^*Tx,x\rangle ).
\end{eqnarray*}
Hence, \begin{eqnarray*}
2\langle e^{{\rm i} \theta}Tx,x\rangle \langle T^*Tx,x\rangle +|\langle Tx,x \rangle |^2+\|Tx\|^4 &=&\left (\langle e^{{\rm i} \theta}Tx,x\rangle+ \langle T^*Tx,x\rangle \right)^2\\
\Rightarrow 2\langle e^{{\rm i} \theta}Tx,x\rangle \langle T^*Tx,x\rangle +|\langle Tx,x \rangle |^2+\|Tx\|^4&=& \langle (e^{{\rm i} \theta}T+  T^*T)x,x\rangle^2\\
\Rightarrow 2|\langle Tx,x\rangle| \langle T^*Tx,x\rangle +|\langle Tx,x \rangle |^2+\|Tx\|^4&\leq& w^2(e^{{\rm i} \theta}T+ T^*T).
\end{eqnarray*}
Therefore, $$2|\langle Tx,x\rangle |~~\langle T^*Tx,x\rangle +|\langle Tx,x \rangle |^2+\|Tx\|^4 \leq \sup_{\theta \in \mathbb{R}} w^2(e^{{\rm i} \theta}T+ T^*T)$$ and so
\[2c(T)m^2(T)+|\langle Tx,x \rangle |^2+\|Tx\|^4 \leq \sup_{\theta \in \mathbb{R}} w^2(e^{{\rm i} \theta}T+ T^*T).\]
Hence, taking supremum over all unit vectors in $\mathcal{H}$, we get 
\[2c(T)m^2(T)+dw^2(T) \leq \sup_{\theta \in \mathbb{R}} w^2(e^{{\rm i} \theta}T+ T^*T).\]
Thus we have the desired inequality of the theorem.
\end{proof}

\begin{remark}\label{rem-2}
	The inequality in \cite[Th. 2.1]{ZS} states that  $$dw^2(T)\leq w^2(|T|^2-T)+2\|T\|^2w(T).$$
 If we consider the  matrix $T=\left(\begin{array}{cc}
	-1 & 0\\
	0 & -2
	\end{array}\right),$ then it follows from Theorem \ref{th-upper2} that $dw^2(T)\leq 34$, whereas \cite[Th. 2.1]{ZS} gives $dw^2(T)\leq 52.$  This shows that the upper bound of $dw(T)$ obtained in Theorem \ref{th-upper2} is better than that obtained in \cite[Th. 2.1]{ZS}.
\end{remark}

In the next theorem we obtain both lower and upper bounds for the Davis-Wielandt radius of bounded linear operators.
	
\begin{theorem}\label{th-upperlower4}
Let $T\in \mathcal{B}(\mathcal{H}).$  Then 
	\begin{eqnarray*}
\frac{1}{2} \sup_{\theta \in \mathbb{R}} \left \{ w^2(e^{\rm i\theta}T+T^*T)+ c^2(e^{\rm i\theta}T-T^*T)\right \} &\leq& dw^2(T)\\ 
 &\leq & \frac{1}{2} \left \{w^2(T+T^*T)+ w^2(T-T^*T)\right \}.  
	\end{eqnarray*}
	\end{theorem}
	
	\begin{proof}
	Let $x\in \mathcal{H}$ with $\|x\|=1.$ Then there exists $\theta \in \mathbb{R}$ such that $|\langle Tx,x \rangle |=e^{{\rm i} \theta}\langle Tx,x\rangle$. Now,
	\begin{eqnarray*}
	|\langle Tx,x \rangle |^2+\|Tx\|^4 &=& \frac{1}{2}\left(|\langle Tx,x \rangle |+\langle Tx, Tx\rangle  \right)^2+ \frac{1}{2}\left(|\langle Tx,x \rangle |-\langle Tx, Tx\rangle  \right)^2\\
	&=& \frac{1}{2}\left(\langle e^{{\rm i} \theta}Tx,x \rangle +\langle T^*Tx, x\rangle  \right)^2+ \frac{1}{2}\left(\langle e^{{\rm i} \theta}Tx,x \rangle |-\langle T^*Tx, x\rangle  \right)^2\\
	&=& \frac{1}{2}\left(\langle (e^{{\rm i} \theta}T+ T^*T)x, x\rangle  \right)^2+ \frac{1}{2}\left(\langle (e^{{\rm i} \theta}T- T^*T)x, x\rangle  \right)^2.\\
\end{eqnarray*}
Hence, 
\[|\langle Tx,x \rangle |^2+\|Tx\|^4\geq  \frac{1}{2}\left(\langle (e^{{\rm i} \theta}T+ T^*T)x, x\rangle  \right)^2+ \frac{1}{2} c^2(e^{{\rm i} \theta}T- T^*T).\]
Taking supremum over all unit vectors in $\mathcal{H}$, we get	
\[dw^2(T) \geq \frac{1}{2}w^2(e^{\rm i\theta}T+T^*T)+\frac{1}{2}c^2(e^{\rm i\theta}T-T^*T).\]
This holds for all $\theta \in \mathbb{R}$, so
		\[dw^2(T)\geq \frac{1}{2} \sup_{\theta \in \mathbb{R}} \left \{ w^2(e^{\rm i\theta}T+T^*T)+ c^2(e^{\rm i\theta}T-T^*T)\right \}. \]
This implies the first inequality of the theorem.
Again
\begin{eqnarray*}
	|\langle Tx,x \rangle |^2+\|Tx\|^4 &=& \frac{1}{2}\left |\langle Tx,x \rangle +\langle Tx, Tx\rangle  \right|^2+ \frac{1}{2}\left|\langle Tx,x \rangle -\langle Tx, Tx\rangle  \right|^2\\
	&=& \frac{1}{2}\left|\langle Tx,x \rangle +\langle T^*Tx, x\rangle  \right|^2+ \frac{1}{2}\left|\langle Tx,x \rangle -\langle T^*Tx, x\rangle  \right|^2\\
	&=& \frac{1}{2}\left|\langle (T+ T^*T)x, x\rangle  \right|^2+ \frac{1}{2}\left|\langle (T- T^*T)x, x\rangle  \right|^2\\
	&\leq& \frac{1}{2} \left \{w^2(T+T^*T)+ w^2(T-T^*T)\right\}.
\end{eqnarray*}
Therefore, taking supremum over all unit vectors in $\mathcal{H}$, we get
	\[dw^2(T) \leq  \frac{1}{2} \left \{w^2(T+T^*T)+ w^2(T-T^*T)\right\}.\] Hence completes the proof.
\end{proof}	
 
	\begin{remark} 
	
We give  operators for which the inequality in Theorem \ref{th-upperlower4} improves on the existing inequalities in \cite[Th. 2.1,Th. 2.7]{ZS} and (\ref{1stbound}). The inequality in \cite[Th. 2.7]{ZS} states that if $ T \in \mathcal{B}(\mathcal{H}) $ then
\begin{eqnarray*}
dw^2(T)&\leq& \frac{1}{2}w(T^2)+\frac{1}{4}w\left (|T|^2+|T^*|^2\right )\\
&& +4w^2(T)\left( 2w^2(T)-c^2(T)+2w(T)\sqrt{w^2(T)-c^2(T)} \right ).
\end{eqnarray*}
	If we consider  $T=\left(\begin{array}{cc}
	1 & 0\\
	0 & -1
	\end{array}\right),$  then from  Theorem \ref{th-upperlower4} we get $dw^2(T)\leq 4$, whereas \cite[Th. 2.1]{ZS} gives $dw^2(T)\leq 6$ and \cite[Th. 2.7]{ZS} gives $dw^2(T)\leq 17.$ 
	Also Theorem \ref{th-upperlower4} gives $dw^2(T)\geq 2$, whereas the first inequality in (\ref{1stbound}) gives $dw^2(T)\geq 1$.
	\end{remark}

	 We now obtain an upper bound for the Davis-Wielandt radius of a bounded linear operator in terms of non-negative continuous functions. To prove this we need the following two  lemmas. First lemma is known as Power-Young inequality and the second one is known as McCarthy inequality.

	\begin{lemma} $($\cite{SMY}$)$\label{lem6}
	Let $a,b\geq 0$ and $\alpha,\beta>1$ such that $\frac{1}{\alpha}+\frac{1}{\beta}=1.$ Then
	\[ab\leq \frac{1}{\alpha}a^{\alpha}+\frac{1}{\beta}b^{\beta}.\]
	\end{lemma}

	\begin{lemma}$($\cite{K}$)$ \label{lem7}
	Let $A\geq 0.$ Then for all $x\in \mathcal{H}$ with $\|x\|=1,$ we have \[\langle Ax,x\rangle^p\leq \langle A^p x,x\rangle, ~~p\geq 1.\]
	\end{lemma}

\begin{lemma}	$($\cite[Th. 5]{K}$)$  \label{lem9}
Let $T\in \mathcal{B}(\mathcal{H})$ and $f,g$ be two non-negative continuous functions on $[0,\infty)$ such that $f(t)g(t)=t,$ for all $t\in [0,\infty).$ Then
$$|\langle Tx,y \rangle|\leq \|f(|T|)x\| \|g(|T^*|)y\|, ~~\mbox{for all} ~~x,y\in \mathcal{H}.$$
	\end{lemma}
	
Now, we are in a position to prove the following inequality for the Davis-Wielandt radius of bounded linear operators.	
	
	\begin{theorem}\label{upper10}
	Let $T\in \mathcal{B}(\mathcal{H})$. Then
	\[dw^2(T)\leq \left\|  \frac{1}{\alpha_1}f_1^{2\alpha_1}(|T|)+\frac{1}{\beta_1}g_1^{2\beta_1}(|T^*|) + \frac{1}{\alpha_2}f_2^{2\alpha_2}(|T^*T|)+\frac{1}{\beta_2}g_2^{2\beta_2}(|T^*T|) \right\|,\] 
	where $\alpha_i,\beta_i>1$ with $\frac{1}{\alpha_i}+\frac{1}{\beta_i}=1$ and $f_i, g_i$ are two non-negative continuous functions on $[0,\infty)$ such that $f_i(t)g_i(t)=t,$ for all $t\in [0,\infty), i=1,2.$
	\end{theorem}

	\begin{proof}
	Let $x\in \mathcal{H}$ with $\|x\|=1$. Then using Lemmas \ref {lem9}, \ref{lem6} and \ref{lem7}, we get
	\begin{eqnarray*}
	&& |\langle Tx,x \rangle |^2+\|Tx\|^4 \\
	&=& |\langle Tx,x \rangle |^2+ \langle T^*Tx,x \rangle^2\\
	&\leq& \langle f_1^2(|T|)x,x \rangle \langle g_1^2(|T^*|)x,x \rangle+ \langle f_2^2(|T^*T|)x,x \rangle \langle g_2^2(|T^*T|)x,x \rangle\\
	&\leq& \frac{1}{\alpha_1}\langle f_1^2(|T|)x,x \rangle^{\alpha_1}+\frac{1}{\beta_1} \langle g_1^2(|T^*|)x,x \rangle^{\beta_1} + \frac{1}{\alpha_2}\langle f_2^2(|T^*T|)x,x \rangle^{\alpha_2}\\
	&&  +\frac{1}{\beta_2} \langle g_2^2(|T^*T|)x,x \rangle^{\beta_2}\\
	&\leq& \frac{1}{\alpha_1}\langle f_1^{2\alpha_1}(|T|)x,x \rangle+\frac{1}{\beta_1} \langle g_1^{2\beta_1}(|T^*|)x,x \rangle + \frac{1}{\alpha_2}\langle f_2^{2\alpha_2}(|T^*T|)x,x \rangle\\
	&& + \frac{1}{\beta_2} \langle g_2^{2\beta_2}(|T^*T|)x,x \rangle\\
	&=& \left \langle \left (\frac{1}{\alpha_1}f_1^{2\alpha_1}(|T|)+\frac{1}{\beta_1}g_1^{2\beta_1}(|T^*|)+\frac{1}{\alpha_2}f_2^{2\alpha_2}(|T^*T|)+\frac{1}{\beta_2}g_2^{2\beta_2}(|T^*T|) \right)x,x \right \rangle\\
	&\leq& \left\| \frac{1}{\alpha_1}f_1^{2\alpha_1}(|T|)+\frac{1}{\beta_1}g_1^{2\beta_1}(|T^*|)+\frac{1}{\alpha_2}f_2^{2\alpha_2}(|T^*T|)+\frac{1}{\beta_2}g_2^{2\beta_2}(|T^*T|) \right\|.
	\end{eqnarray*}
	Therefore, taking supremum over all unit vectors in $\mathcal{H}$, we get the required inequality.
	\end{proof}

	As a consequence of  Theorem \ref{upper10}, we get the following corollary.
	
	\begin{cor}\label{upper8}
	Let $T\in \mathcal{B}(\mathcal{H}).$ Then the following inequalities are hold.\\
	$ (i) $ If $\alpha,\beta>1$ such that $\frac{1}{\alpha}+\frac{1}{\beta}=1$ then we have
	 \[ dw^2(T) \leq  \left \| \frac{1}{\alpha}|T|^{\alpha}\left(1+|T|^{\alpha}\right)+\frac{1}{\beta}\left(|T^*|^{\beta}+|T|^{2\beta}\right) \right \|.\]
	$ (ii) ~~\hspace{1.8cm}  dw^2(T) \leq  \frac{1}{2}\left \| |T|^2+|T^*|^2  + 2|T|^4 \right \|.$
	%$$ (ii)  \hspace{4cm} dw^2(T) \leq  \frac{1}{2}\left \| |T|^2+|T^*|^2  + 2|T|^4 \right \|.\hspace{2cm}$$
	\end{cor}
	
	\begin{proof}
		If we take $f_1(t)=g_1(t)=f_2(t)=g_2(t)=t^{\frac{1}{2}}$ and $\alpha_1=\alpha_2=\alpha, \beta_1=\beta_2=\beta$ in Theorem \ref{upper10}, then we get $(i).$ Further taking $\alpha=\beta=2$ in (i), we get $(ii).$
	\end{proof}

\begin{remark}	
We would like to note that the inequality in Corollary \ref{upper8} (ii) becomes equality if $T$ is normaloid.
\end{remark}

Next, we need the following lemma which is a generalization of Cauchy-Schwarz inequality, the proof of which can be found in \cite{B}.	
	
\begin{lemma}\label{lem13}
Let $a,b,c \in \mathcal{H}$ with $\|c\|=1$. Then 
\[|\langle a,c\rangle \langle c,b \rangle|\leq \frac{1}{2}\left (|\langle a,b \rangle |+\|a\|\|b\| \right ).\]
\end{lemma}

Using Lemma \ref{lem13}, we prove the following inequalities.

\begin{theorem}\label{upper15}
Let $T\in \mathcal{B}(\mathcal{H}).$ Then the following inequalities are hold.
\begin{eqnarray*}
(i) && dw^2(T)\leq \left \||T|^2+|T|^4\right \|,\\
(ii) && dw^2(T)\leq \frac{1}{2}\left(w(T^2)+\|T\|^2 \right)+\left \| T \right\|^4.
\end{eqnarray*}
\end{theorem}	

\begin{proof}
Let $x\in \mathcal{H}$ with $\|x\|=1$. Then by using Lemma \ref{lem13}, we get 
\begin{eqnarray*}
	|\langle Tx,x \rangle |^2+\|Tx\|^4 & = & |\langle Tx,x \rangle \langle x,Tx \rangle| + \langle |T|^2x,x \rangle \langle x,|T|^2x\rangle \\ 
	&\leq & \frac{1}{2}(\|Tx\|^2+ \langle Tx,Tx \rangle ) + \frac{1}{2}(\||T|^2x\|^2+\langle |T|^2x,|T|^2x \rangle ) \\
	&= & \langle |T|^2x,x \rangle +\langle |T|^4x,x \rangle \\
	&= & \langle (|T|^2+|T|^4)x,x \rangle.
\end{eqnarray*}
Therefore, taking supremum over all unit vectors in $\mathcal{H}$, we get the inequality (i). Again considering $|\langle Tx,x \rangle |^2 = |\langle Tx,x \rangle \langle x,T^*x \rangle|$, $|\langle Tx,Tx \rangle |^2 = \langle |T|^2x,x \rangle \langle x,|T|^2x \rangle $ and then using Lemma \ref{lem13}, we get the inequality (ii).
\end{proof}

\begin{remark}	
It is easy to see that both the inequalities in Theorem \ref{upper15} becomes equality if $T$ is normaloid.
\end{remark}

 In the following remark, we show that there exist operators for which the inequalities in Theorem \ref{upper15} are better than the existing inequalities in \cite{ZS}. 

\begin{remark} The inequalities in \cite[Th. 2.13,Th. 2.14, Th. 2.16, Th. 2.17]{ZS} are noted respectively as follows
\begin{eqnarray*}
dw^2(T)&\leq& \max\{ \|T\|^2, \|T\|^4\}+\sqrt{2}w(|T|^2T),\\
dw^2(T)&\leq& \frac{1}{2}\left ( w(|T|^4+|T|^2)+w(|T|^4-|T|^2)\right)+\sqrt{2}w(|T|^2T),\\
dw^2(T)&\leq& \max\{w(T),w(|T|^2) \}\left ( w(|T|^4+|T|^2)+2 w(|T|^2T) \right )^{\frac{1}{2}},\\
dw^2(T)&\leq& \|T\| \max\{w(T),w(|T|^2) \} \left ( 1+\|T\|^2+2w(T)  \right )^{\frac{1}{2}}.
\end{eqnarray*}
 If we take $T=\left(\begin{array}{cc}
	1 & 1\\
	0 & 0
	\end{array}\right)$  then from Theorem \ref{upper15} (i) and Theorem \ref{upper15} (ii), we get  $dw^2(T)\leq 6$ and $ 5.6$, respectively, whereas the inequalities in \cite[Th. 2.1, Th. 2.7, Th. 2.13, Th. 2.14, Th. 2.16, Th. 2.17]{ZS} respectively gives  $dw^2(T)\leq 6.283, 35.416, 6.828,6.828,  $ $ 6.325, 6.58$.
	Thus, the bounds of $dw(T)$ obtained in Theorem \ref{upper15} (i) and Theorem \ref{upper15} (ii) are better than the existing ones.
\end{remark}

Next, we obtain an upper bound for the Davis-Wielandt radius of a bounded linear operator using the following lemma, the proof of which can be found in \cite[Lemma 2.1]{Dra}.

\begin{lemma}$($\cite[Lemma 2.1]{Dra}$)$ \label{lem18}
Let $x,y \in \mathcal{H}$ and  $\lambda \in \mathbb{C}$. Then the following equality holds:
\[\|x\|^2\|y\|^2-|\langle x,y\rangle|^2=\|x-\lambda y\|^2\|y\|^2-|\langle x-\lambda y,y\rangle|^2.\]
\end{lemma}

\begin{theorem}\label{upper20}
Let $T\in \mathcal{B}(\mathcal{H}).$ Then 
\begin{eqnarray*}
 dw^2(T) &\leq& \inf_{\lambda \in \mathbb{C}}\Big \{\Big (2 \left \|Re(\lambda)~Re(T)+Im(\lambda)~Im(T)\right \|   +\left \|T^*T-2Re(\overline{\lambda}T)\right \|\Big )^2 \\
 && + 2 \|Re(\overline{\lambda}T)\|- |\lambda|^2+w^2(T-\lambda I) \Big \}.
\end{eqnarray*}
In particular, $dw(T) \leq \sqrt{w^2(T) + \|T\|^4}$.
\end{theorem}

\begin{proof}
Let $x\in \mathcal{H}$ with $\|x\|=1.$ Let $\lambda \in \mathbb{C}$. By using Lemma \ref{lem18}, we have 
\begin{eqnarray*}
\|Tx\|^2||x||^2-|\langle Tx,x \rangle|^2 &=& \|Tx-\lambda x\|^2\|x\|^2-|\langle Tx-\lambda x ,x \rangle|^2.
\end{eqnarray*}
Using the Cartesian decomposition of $T$, i.e.,  $T=Re(T)+{\rm i}~Im(T)$, we get
\begin{eqnarray*}
 \|Tx\|^2 &=& \left (\langle Re(T)x,x \rangle \right)^2 -\left (\langle Re(T-\lambda I)x,x \rangle \right )^2+ \left (\langle Im(T)x,x \rangle \right)^2\\
&& - \left(\langle Im(T-\lambda I)x,x \rangle \right)^2+\|Tx-\lambda x\|^2\\
&=& \langle (2Re(T)-Re(\lambda)I)x,x\rangle \langle Re(\lambda) x,x \rangle \\
&&+\langle ( 2Im(T)-Im(\lambda)I)x,x\rangle \langle Im(\lambda) x,x \rangle +\|Tx-\lambda x\|^2\\
&=& 2Re(\lambda) \langle Re(T)x,x \rangle +2Im(\lambda) \langle Im(T)x,x \rangle \\
&&- (Re(\lambda))^2-(Im(\lambda))^2+\|Tx-\lambda x\|^2\\
&=& 2 \left (Re(\lambda)\langle Re(T)x,x \rangle+Im(\lambda) \langle Im(T)x,x \rangle \right)-|\lambda|^2\\
&& + \left \langle Tx-\lambda x ,Tx-\lambda x \right \rangle\\
&=& 2 \left (Re(\lambda)\langle Re(T)x,x \rangle+Im(\lambda) \langle Im(T)x,x \rangle \right)\\
&& + \left \langle (T^*T-2Re(\overline{\lambda}T))x,x \right \rangle\\
&\leq& 2 \left \|Re(\lambda)~Re(T)+Im(\lambda)~Im(T)\right \|   +\left \|T^*T-2Re(\overline{\lambda}T)\right \|.
\end{eqnarray*}
Again by using Lemma \ref{lem18}, we get
 \begin{eqnarray*}
|\langle Tx,x \rangle|^2 &=& \|Tx\|^2-\|Tx-\lambda x\|^2+|\langle Tx-\lambda x ,x \rangle|^2\\
&=& 2\langle Re(\overline{\lambda}T)x,x \rangle - |\lambda|^2+|\langle Tx-\lambda x ,x \rangle|^2\\
&\leq& 2 \|Re(\overline{\lambda}T)\|- |\lambda|^2+w^2(T-\lambda I).
\end{eqnarray*}
Hence, 
\begin{eqnarray*}
&& |\langle Tx,x \rangle|^2+\|Tx\|^4\\ 
&&\leq 2 \|Re(\overline{\lambda}T)\|- |\lambda|^2+w^2(T-\lambda I)\\
&&+  \left (2 \left \|Re(\lambda)~Re(T)+Im(\lambda)~Im(T)\right \|   +\left \|T^*T-2Re(\overline{\lambda}T)\right \| \right )^2.
\end{eqnarray*}
Therefore, taking supremum over all unit vectors in $\mathcal{H}$, and infimum over all $\lambda \in \mathbb{C}$, we get
\begin{eqnarray*}
 dw^2(T) &\leq& \inf_{\lambda \in \mathbb{C}}\Big \{\left (2 \left \|Re(\lambda)~Re(T)+Im(\lambda)~Im(T)\right \|   +\left \|T^*T-2Re(\overline{\lambda}T)\right \|\right )^2 \\
 && + 2 \|Re(\overline{\lambda}T)\|- |\lambda|^2+w^2(T-\lambda I) \Big \}.
\end{eqnarray*}
Taking $\lambda =0$, we get $dw(T) \leq \sqrt{w^2(T) + \|T\|^4}.$
\end{proof}

In the following remark, we show that the inequality in Theorem \ref{upper20} improves on the existing inequalities in \cite{ZS} for some operators.

	\begin{remark} In \cite[Th. 2.2]{ZS}, Zamani and Shebrawi proved that 
	$$dw^2(T)\leq \frac{1}{2}w(|T|^2+2|T|^4+|T^*|^2)-\frac{1}{2} \inf_{\|x\|=1}(\|Tx\|-\|T^*x\|)^2.$$
	If we consider a matrix $T=\left(\begin{array}{cc}
	0 & 2\\
	0 & 0
	\end{array}\right)$ then Theorem \ref{upper20} gives $dw(T)\leq  4.123$, whereas \cite[Th. 2.1]{ZS} gives $dw(T)\leq 5.0935$, \cite[Th. 2.2]{ZS} gives $dw(T)\leq 4.2426$ and \cite[Th. 2.17]{ZS} gives $dw(T)\leq 4.6006$. 
	\end{remark}

We next  obtain an upper bound for the Davis-Wielandt radius of sum of two bounded linear operators. 

\begin{theorem}\label{lem24}
	Let $ S,T \in \mathcal{B}(\mathcal{H})$. Then
	\[dw(S+T)\leq dw(S)+dw(T)+w(S^*T+T^*S).\]
\end{theorem}

\begin{proof}
	From the definition of the Davis-Wielandt shell, we have
	\begin{eqnarray*}
		DW(S+T) &=& \left \{ \Big (\left \langle (S+T)x,x \right \rangle , \left \langle (S+T)x,(S+T)x \right \rangle \Big) : x \in \mathcal{H},\|x\|=1 \right \}\\
		&=& \Big\{ \Big( \langle Sx,x \rangle , \langle Sx,Sx \rangle \Big)+ \Big( \langle Tx,x \rangle , \langle Tx,Tx \rangle \Big)\\
		&& +\Big(0,\langle (S^*T+T^*S)x ,x\rangle \Big) : x \in \mathcal{H},\|x\|=1 \Big \}.
		\end {eqnarray*}
		Hence, $DW(S+T)\subseteq DW(S)+DW(T)+A$, where
		\[A = \left \{ \left (0,\langle (S^*T+T^*S)x ,x\rangle \right): x \in \mathcal{H},\|x\|=1 \right \}. \]
		This implies the required inequality of the theorem.
	\end{proof}

	The following Corollary follows from Theorem \ref{lem24}. 
	
	\begin{cor}\label{prop25}
		Let $S,T \in \mathcal{B}(\mathcal{H})$ be such that $S^*T+T^*S=0$. Then 
		\[dw(S+T)\leq dw(S)+dw(T).\]
	\end{cor}

Our next result in this section is the estimation of upper and lower bounds for the Davis-Wielandt radius of the shift operator on $\mathbb{C}^n$. 

\begin{theorem}\label{33}
Let $T$ be the right shift operator on $\mathbb{C}^n$ defined by $T=(t_{ij})_{n\times n}$, where 
$t_{ij}= \begin{cases}
1 &  j=i-1  \\
0 &  j \neq i-1.
\end{cases}$ 
Then $$\sqrt{\cos^2\left( \frac{\pi}{n}\right)+1} \leq dw(T) \leq \sqrt{\cos^2\left( \frac{\pi}{n+1}\right)+1}.$$
\end{theorem}

\begin{proof}
Clearly, $\|T\|=1$ and from \cite [p. 8]{GR}, we have  $w(T)=\cos (\frac{\pi}{n+1})$.
 Therefore, $dw(T) \leq \sqrt{w^2(T)+\|T\|^4}=\sqrt{\cos^2 (\frac{\pi}{n+1})+1}$. This is the second inequality. 

\noindent To prove the first inequality, let $f=(f_1,f_2,...,f_n) \in \mathbb{C}^n $ with $\|f\|=1$, i.e.,  $|f_1|^2+|f_2|^2+...+|f_n|^2=1$. Then we have,  $Tf=(0,f_1,f_2,...,f_{n-1}),
\langle Tf,f \rangle = f_1\overline{f_2}+f_2\overline{f_3}+...+f_{n-1}\overline{f_n} $ and 
$\langle Tf,Tf \rangle = |f_1|^2+|f_2|^2+...+|f_{n-1}|^2.$
In particular, if we choose $f_n=0$, then we get $\langle Tf,Tf \rangle=1$ and  $|\langle Tf,f \rangle| = |f_1\overline{f_2}+f_2\overline{f_3}+...+f_{n-2}\overline{f_{n-1}}|$. Also, $\sup_{\|f\|=1}\left \{|f_1\overline{f_2}+f_2\overline{f_3}+...+f_{n-2}\overline{f_{n-1}}| \right \}=\cos^2( \frac{\pi}{n}),$ (see \cite [p. 8]{GR}). Therefore,
\[ dw(T) \geq \sup_{\|f\|=1}\sqrt{|\langle Tf,f\rangle|^2+\|Tf\|^4 }= \sqrt{\cos^2 \left(\frac{\pi}{n}\right)+1}.\]
This is the first inequality of the theorem and it completes the proof.
\end{proof}

\begin{remark} 1. If we consider the left shift operator $T$ on $\mathbb{C}^n$ defined by $T=(t_{ij})_{n\times n}$, where 
$t_{ij}= \begin{cases}
1 &  j=i+1  \\
0 &  j \neq i+1,
\end{cases}$ then similarly as in Theorem \ref{33},	we can prove that 
 $$\sqrt{\cos^2\left( \frac{\pi}{n}\right)+1} \leq dw(T) \leq \sqrt{\cos^2\left( \frac{\pi}{n+1}\right)+1}.$$
 2. If $T$ is a shift operator on the Hilbert space $ \ell_2$  then $T$ is normaloid and so  $ dw(T) = \sqrt{2}.$
\end{remark}

\noindent \textbf{Comparabilty of bounds}\\
	
 Before we end this section we would like to discuss about the comparability of the bounds for the Davis-Wielandt radius obtained by us. We begin with the following observations:
\begin{itemize}
\item If $c(T)=0$ or $m(T)=0$ then the second inequality in Theorem \ref{th-upperlower4} gives better bound than the inequality in Theorem \ref{th-upper2}.
\item  If $T^2=0$ then the inequality in Corollary \ref{upper8} (ii) gives better bound than the inequality in Theorem \ref{upper15} (ii).
\item  If $T$ is normaloid, i.e., if $w(T)=\|T\|$ then the inequalities in Corollary \ref{upper8} (ii), Theorem \ref{upper15} (i), Theorem \ref{upper15} (ii) and Theorem \ref{upper20} give equalities and so the upper bounds obtained in those theorems are better than that obtained  in Theorem \ref{th-upper2} and Theorem \ref{th-upperlower4} for a normaloid operator.
\end{itemize}
Next, we  exhibit numerical examples to  illustrate that the upper bounds for the Davis-Wielandt radius obtained by us  are, in general,  incomparable. We cite a few of such examples.

\begin{itemize}
\item 	Let $T_1=\left(\begin{array}{cc}
	0 & 1\\
	2 & 0
	\end{array}\right).$  Then the upper bound obtained in Theorem \ref{th-upperlower4} is  21.357, whereas that in Theorem \ref{upper15} (ii) is 19. On the other hand, if we consider  $T_2=\left(\begin{array}{cc}
	0 & 2\\
	0 & 0
	\end{array}\right),$ then  the upper bound obtained in Theorem \ref{upper15} (ii) is 18, whereas that in Theorem \ref{th-upperlower4} is 17.944.
	
\item  Let $T_3=\left(\begin{array}{cc}
1 & 1\\
0 & 0
\end{array}\right).$  Then the upper bound obtained in  Corollary \ref{upper8} (ii) is 5.5495, whereas that obtained in  Theorem \ref{upper15} (ii) is 5.6.   On the other hand, if we consider  $T_4=\left(\begin{array}{cc}
1 & 1\\
0 & 1
\end{array}\right),$  then the upper bound obtained in  Corollary \ref{upper8} (ii) is 9.272, whereas that obtained in  Theorem \ref{upper15} (ii) is 9.162. 

\item The following self-explanatory table further illustrates the incomparability of the bounds for the Davis-Wielandt radius obtained by us. 

\begin{center}
	\begin{tabular} {|l|c|c|c|c|}
		\hline
		 & \multicolumn{4}{|c|}{Upper bound for }       \\
		% & \multicolumn{4}{|c|}{Davis-Wielandt radius for the operator}\\
		 \hline
		&  $dw^2(T_1)$ &  $dw^2(T_2)$ &  $dw^2(T_3)$ &   $dw^2(T_4)$\\
		\hline
		\hline
		Th. \ref{th-upperlower4} & 21.357& 17.944 & 5.4753  & 9.056      \\
		\hline
		Cor. \ref{upper8} (ii) & 18.5& 18   & 5.5495  & 9.272  \\
		\hline
		Th. \ref{upper15} (i) & 20& 20  & 6  & 9.472\\
		\hline
		Th. \ref{upper15} (ii) & 19& 18 & 5.6 & 9.162\\
		\hline
	 Th. \ref{upper20} & 18.25& 17 & 5.4568  & 9.104     \\
		\hline
	\end{tabular}
\end{center}
\end{itemize}

\section{\textbf{Davis-Wielandt radius inequalities of operator matrices}}

\smallskip 

\noindent In this section we obtain some estimations for the Davis-Wielandt radius of $2\times 2$ operator matrices. We determine the exact value for the Davis-Wielandt radius of  operator matrices of the form $\left(\begin{array}{cc}
	I & B\\
	0 & 0
	\end{array}\right)$ and $\left(\begin{array}{cc}
	0 & A\\
	B & 0
	\end{array}\right)$.  We also obtain bounds for the Davis-Wielandt radius of $\left(\begin{array}{cc}
	A & B\\
	0 & C
	\end{array}\right)$. To achieve our goal, we need the following lemmas. First lemma follows from $DW(U^*TU)=DW(T)$ for every unitary operator $U \in \mathcal{B}(\mathcal{H}).$

	\begin{lemma}\label{lem22}
	Let $T \in \mathcal{B}(\mathcal{H})$. Then for every unitary operator $U \in \mathcal{B}(\mathcal{H}),$ we have
	$$dw(U^*TU)=dw(T).$$
	\end{lemma}

\begin{lemma}\label{lem23}
	Let $A,B \in \mathcal{B}(\mathcal{H})$. Then 
\begin{enumerate}[label=$(\alph*)$]
\item $dw\left(\begin{array}{cc}
	0 & A\\
	e^{{\rm i} \theta}B & 0
	\end{array}\right)=dw\left(\begin{array}{cc}
	0 & A\\
	B & 0
	\end{array}\right)$, for every $\theta \in \mathbb{R}.$
	\item $dw\left(\begin{array}{cc}
	0 & A\\
	B & 0
	\end{array}\right)=dw\left(\begin{array}{cc}
	0 & B\\
	A & 0
	\end{array}\right).$ 
	\item $dw\left(\begin{array}{cc}
	A & B\\
	B & A
	\end{array}\right)=dw\left(\begin{array}{cc}
	A-B & 0\\
	0 & A+B
	\end{array}\right).$
	\item $dw\left(\begin{array}{cc}
	A & 0\\
	0 & B
	\end{array}\right)=dw\left(\begin{array}{cc}
	B & 0\\
	0 & A
	\end{array}\right).$ 
	\end{enumerate}
	\end{lemma}

\begin{proof}
\begin{enumerate}[label=(\alph*)]
\item Let $U=\left(\begin{array}{cc}
	I & 0\\
	0 & e^{{\rm i}\frac{\theta}{2}}I
	\end{array}\right)$. Then by using Lemma \ref{lem22}, we get\\
	$dw\left(\begin{array}{cc}
	0 & A\\
	e^{{\rm i} \theta}B & 0
	\end{array}\right)=dw\left(U^* \left(\begin{array}{cc}
	0 & A\\
	e^{{\rm i}\theta}B & 0
	\end{array} \right) U\right)=dw\left(\begin{array}{cc}
	0 & e^{{\rm i}\frac{\theta}{2}} A\\
	e^{{\rm i}\frac{\theta}{2}} B & 0
	\end{array}\right) = dw\left(\begin{array}{cc}
	0 & A\\
	B & 0
	\end{array}\right).$
	
\item Let $U=\left(\begin{array}{cc}
	0 & I\\
	I & 0
	\end{array}\right)$. Then by using Lemma \ref{lem22}, we get (b).
	
	\item Let $U=\frac{1}{\sqrt{2}}\left(\begin{array}{cc}
	I & I\\
	-I & I
	\end{array}\right)$. Then by using Lemma \ref{lem22}, we get (C). 
	
	\item Let $U=\left(\begin{array}{cc}
	0 & I\\
	I & 0
	\end{array}\right)$. Then by using Lemma \ref{lem22}, we get (d).
\end{enumerate}
\end{proof}

\begin{lemma}\label{33}
Let $A,B \in \mathcal{B}(\mathcal{H}).$ Then
$$dw\left(\begin{array}{cc}
	A & 0\\
	0 & B
	\end{array}\right) = \max \Big\{dw(A),dw(B) \Big\}.$$
\end{lemma}

\begin{proof}
Let $T=\left(\begin{array}{cc}
	A & 0\\
	0 & B
	\end{array}\right).$ Let $x$ be a unit vector in $\mathcal{H}$ and let $\tilde{x} =\left(\begin{array}{c}
	x \\
	0 
	\end{array}\right) \in \mathcal{H} \oplus \mathcal{H}. $ Clearly $\|\tilde{x}\|=1$, therefore we have 
	$$ |\langle Ax,x \rangle|^2+ \|Ax\|^4 =  |\langle T\tilde{x},\tilde{x} \rangle|^2+ \|T\tilde{x}\|^4 \leq dw^2(T) .$$
	Taking supremum over all unit vectors in $\mathcal{H}$, we get, $dw^2(A) \leq dw^2(T).$
	Similarly, we can prove that, $dw^2(B) \leq dw^2(T).$
	Combining above two inequalities, we get $$\max \{dw(A),dw(B) \} \leq dw(T).$$ 
	To complete the proof of the lemma, we only need to show, $ dw(T) \leq \max \{dw(A),dw(B) \} $. Let  $z=\left(\begin{array}{cc}
	x \\
	y 
	\end{array}\right) \in \mathcal{H}\oplus \mathcal{H} $ be such that $\|z\|=1$, i.e., $\|x\|^2+\|y\|^2=1$.	Then
\begin{eqnarray*}
|\langle Tz,z \rangle|^2+\|Tz\|^4 &=& |\langle Ax,x \rangle+\langle By,y \rangle|^2 +(\|Ax\|^2+\|By\|^2)^2\\
&\leq& (|\langle Ax,x \rangle|+|\langle By,y \rangle|)^2 +(\|Ax\|^2+\|By\|^2)^2\\
&\leq& \left(\sqrt{|\langle Ax,x \rangle|^2+\|Ax\|^4}+\sqrt{|\langle By,y \rangle|^2+\|By\|^4}\right)^2,\\
&& \,\,\,\,\,\,\,\,\,\,\,\,\,\,\,\,\,\,\,\,\,\,\,\,\,\,\,\,\,\,\,\,\,\,\,\,\,\,\,\,\,\,\,\,\,\,\,\,\,\,\,\,\,\,\,\,\,\,\,\,\,\,\,\,\,\,\,\,\,\,\,\,\,\,\,\,\,\,\,\,\,~~\mbox{by Minkowski inequality}\\
&\leq& \left( dw(A) \|x\|^2 + dw(B)\|y\|^2 \right)^2\\
&\leq&  \max\{dw^2(A),dw^2(B) \}.
\end{eqnarray*}	
Taking supremum over all unit vectors in $\mathcal{H} \oplus \mathcal{H}$, we get $$dw^2(T) \leq \max\{dw^2(A),dw^2(B) \},~~ \mbox{i.e.,}~~dw(T) \leq \max \{dw(A),dw(B) \}.$$
\end{proof}

Now using Lemma \ref{lem23} $(a),(c)$ and Lemma \ref{33}, we have the following proposition:

\begin{prop}\label{eql}
Let $B \in \mathcal{B}(\mathcal{H})$ and let $\theta \in \mathbb{R}.$ Then 
$$dw\left(\begin{array}{cc}
	0 & B\\
	e^{{\rm i} \theta}B & 0
	\end{array}\right)=dw(B).$$
\end{prop}

In the next two theorems, we compute the  exact value of the Davis-Wielandt radius for two special type of operator matrices $\left(\begin{array}{cc}
I & B\\
0 & 0
\end{array}\right)$ and $\left(\begin{array}{cc}
0 & B\\
0 & 0
\end{array}\right)$, where $B\in \mathcal{B}(\mathcal{H})$.

\begin{theorem}\label{31}
Let $B \in \mathcal{B}(\mathcal{H})$ and  $T =\left(\begin{array}{cc}
	I & B\\
	0 & 0
	\end{array}\right) \in \mathcal{B}(\mathcal{H} \oplus \mathcal{H})$. Then 
	
		\[ dw\left(\begin{array}{cc}
	I & B\\
	0 & 0
	\end{array}\right)=\begin{cases}
	\sqrt{2}, &B=0 \\
	(\cos \theta_0 + \|B\| \sin \theta_0)(\cos^2 \theta_0 +(\cos \theta_0 + \|B\|\sin \theta_0)^2)^\frac{1}{2}, & B \neq 0,\\
	\end{cases} \]	
where $b=\|B\|$, $p=-\frac{2b^2-5}{2b},$ $q =- \frac{2b^2-2}{b^2},$ $r= -\frac{3}{2b},$ $s= \frac{1}{2^43^3b^6}(8b^8+20b^6+45b^4+61b^2+28),$ $\alpha=\frac{1}{27}(2p^3-9pq+27r),$ $\beta =(-\frac{\alpha}{2}+\sqrt{s})^\frac{1}{3},$ $\gamma=(-\frac{\alpha}{2}-\sqrt{s})^\frac{1}{3}$ and $\theta_0 = \tan^{-1}(\beta+\gamma-\frac{p}{3}).$

\end{theorem}

\begin{proof}
The proof for the case $B=0$ follows trivially. So we consider  $B\neq 0$.\\
Let  $z=\left(\begin{array}{cc}
	x \\
	y 
	\end{array}\right) \in \mathcal{H}\oplus \mathcal{H} $ be such that $\|z\|=1$, i.e., $\|x\|^2+\|y\|^2=1$.
	Then $ \langle Tz,z \rangle = \langle x+By,x \rangle ~~\mbox{and} ~~ \langle Tz,Tz \rangle = \langle x+By,x+By \rangle. $
	Now, we have
	\begin{eqnarray*}
	|\langle Tz,z \rangle|^2+|\langle Tz,Tz \rangle|^2 &\leq& \|x+By\|^2\|x\|^2 +\|x+By\|^4 \\
	&=& \|x+By\|^2 \left (\|x\|^2 +\|x+By\|^2 \right)\\
	&\leq& \sup_{\|x\|^2+\|y\|^2=1}(\|x\|+\|B\|\|y\|)^2(\|x\|^2+(\|x\|+\|B\|\|y\|)^2)\\
	&=& \sup_{\theta \in [0,\frac{\pi}{2}]}(\cos \theta+\|B\|\sin \theta)^2(\cos^2 \theta+(\cos \theta+\|B\|\sin \theta)^2)\\
	&=& (\cos \theta_0 + \|B\| \sin \theta_0)^2(\cos^2 \theta_0 +(\cos \theta_0 + \|B\|\sin \theta_0)^2),
	\end{eqnarray*}
where $b=\|B\|$, $p=-\frac{2b^2-5}{2b},$ $q =- \frac{2b^2-2}{b^2},$ $r= -\frac{3}{2b},$ $s= \frac{1}{2^43^3b^6}(8b^8+20b^6+45b^4+61b^2+28),$ $\alpha=\frac{1}{27}(2p^3-9pq+27r),$ $\beta =(-\frac{\alpha}{2}+\sqrt{s})^\frac{1}{3},$ $\gamma=(-\frac{\alpha}{2}-\sqrt{s})^\frac{1}{3}$ and $\theta_0 = \tan^{-1}(\beta+\gamma-\frac{p}{3}).$\\
Therefore, taking supremum over all unit vectors $z \in \mathcal{H} \oplus \mathcal{H}$, we get  \\
\[dw(T) \leq (\cos \theta_0 + \|B\| \sin \theta_0)(\cos^2 \theta_0 +(\cos \theta_0 + \|B\|\sin \theta_0)^2)^\frac{1}{2} .\]
 We now show that there exists a sequence $\{z_n\}$ in $\mathcal{H} \oplus \mathcal{H}$ with $\|z_n\|=1$ such that $ \lim_{n \to \infty}(|\langle Tz_n,z_n \rangle|^2+|\langle Tz_n,Tz_n \rangle|^2)^\frac{1}{2}=(\cos \theta_0 + \|B\| \sin \theta_0)(\cos^2 \theta_0 +(\cos \theta_0 + \|B\|\sin \theta_0)^2)^\frac{1}{2}.$
 Since $B\in \mathcal{B}(\mathcal{H})$, there exists a sequence $\{y_n\}$ in $\mathcal{H}$ with $\|y_n\|=1$ such that $\lim_{n \to \infty} \|By_n\| = \|B\|.$
 Let $z^k_{n}=\frac{1}{\sqrt{\|By_n\|^2+k^2}}\left(\begin{array}{cc}
	By_n \\
	ky_n 
	\end{array}\right)$, where $k \geq 0 $. Then
	$|\langle Tz^k_{n},z^k_{n} \rangle|^2+|\langle Tz^k_{n},Tz^k_{n} \rangle|^2  = \frac{(1+k)^2\|By_n\|^4}{(\|By_n\|^2+k^2)^2}\left (1+(1+k)^2\right)$

	$= \left (\frac{\|By_n\|}{\sqrt{\|By_n\|^2+k^2}}+\frac{k\|By_n\|}{\sqrt{\|By_n\|^2+k^2}} \right )^2
	 \left (\frac{\|By_n\|^2}{\|By_n\|^2+k^2}+\left (\frac{\|By_n\|}{\sqrt{\|By_n\|^2+k^2}}+\frac{k\|By_n\|}{\sqrt{\|By_n\|^2+k^2}} \right )^2 \right ).$
We can choose $k_0 \geq 0$ such that  $\frac{\|B\|}{\sqrt{\|B\|^2+k_0^2}} = \cos \theta_0$ and  $ \frac{k_0}{\sqrt{\|B\|^2+k_0^2}}= \sin \theta_0$.
Therefore, if we choose $z_n=\frac{1}{\sqrt{\|By_n\|^2+k_0^2}}\left(\begin{array}{cc}
	By_n \\
	k_0y_n 
	\end{array}\right)$, then $\lim_{n \to \infty}(|\langle Tz_n,z_n \rangle|^2+|\langle Tz_n,Tz_n \rangle|^2)^\frac{1}{2}$ $= \Big (\cos \theta_0 + \|B\| \sin \theta_0 \Big)\Big (\cos^2 \theta_0 +(\cos \theta_0 + \|B\|\sin \theta_0)^2\Big)^\frac{1}{2}.$ 
This completes the proof.
\end{proof}

\begin{example}
Let $B =\left(\begin{array}{cc}
	0 & 2\\
	0 & 0
	\end{array}\right)$ and $ T = \left(\begin{array}{cc}
	I & B\\
	0 & 0
	\end{array}\right).$ Then $b=2$, $ p= -\frac{3}{4},$ $ q=-\frac{3}{2},$ $ r=-\frac{3}{4},$ $s=0.15625,$ $\alpha=-1.15625,$ $\beta=0.991,$ $\gamma=0.5676, $ and $\theta_0=1.0657 $. Therefore from Theorem \ref{31}, we have  $dw(T) = dw \left(\begin{array}{cccc}
	1 & 0 & 0 & 2\\
	0 & 1 & 0 & 0\\
	0 & 0 & 0 & 0\\
	0 & 0 & 0 & 0
	\end{array}\right) = 5.107.$
\end{example}

\begin{theorem}\label{32}
Let $B \in \mathcal{B}(\mathcal{H})$ and  $T =\left(\begin{array}{cc}
	0 & B\\
	0 & 0
	\end{array}\right) \in \mathcal{B}(\mathcal{H} \oplus \mathcal{H})$. Then 
	\[ dw\left(\begin{array}{cc}
	0 & B\\
	0 & 0
	\end{array}\right)=\begin{cases}
		0, &B=0 \\
		\frac{\|B\|}{2\sqrt{1-\|B\|^2}}, & \|B\| < \frac{1}{\sqrt{2}} \\
		\|B\|^2, & \|B\| \geq \frac{1}{\sqrt{2}}.
		\end{cases} \]	
\end{theorem}
\begin{proof}
The proof for the case $B=0$ follows trivially. So we consider  $B\neq 0$.\\
 Let  $z=\left(\begin{array}{cc}
	x \\
	y 
	\end{array}\right) \in \mathcal{H}\oplus \mathcal{H} $ be such that $\|z\|=1$, i.e., $\|x\|^2+\|y\|^2=1$.
Then $ \langle Tz,z \rangle = \langle By,x \rangle ~~\mbox{and} ~~ \langle Tz,Tz \rangle = \langle By,By \rangle.$ Now we have
\begin{eqnarray*}
	|\langle Tz,z \rangle|^2+|\langle Tz,Tz \rangle|^2 &\leq& \|By\|^2\|x\|^2 +\|By\|^4 \\
	&\leq & \sup_{\|x\|^2+\|y\|^2=1} \left (\|B\|^2\|y\|^2\|x\|^2 + \|B\|^4\|y\|^4 \right)\\
	&=& \sup_{\theta \in [0,\frac{\pi}{2}]} \|B\|^2 \sin^2 \theta \left ( \cos^2 \theta +\|B\|^2 \sin^2 \theta \right )\\ 
	\end{eqnarray*}
	First we assume that  $ 0 < \|B\| < \frac{1}{\sqrt{2}}.$
	Then  $$\sup_{\theta \in [0,\frac{\pi}{2}]} \|B\|^2 \sin^2 \theta \left ( \cos^2 \theta +\|B\|^2 \sin^2 \theta \right )=\frac{\|B\|^2}{4(1-\|B\|^2)}.$$
Therefore, $dw(T) \leq \frac{\|B\|}{2\sqrt{(1-\|B\|^2)}}.$ We show that there exist a sequence $\{z_n\}$ in $\mathcal{H} \oplus \mathcal{H}$ with $\|z_n\|=1$ such that 
$$ \lim_{n \to \infty}\{|\langle Tz_n,z_n \rangle|^2+|\langle Tz_n,Tz_n \rangle|^2\}^\frac{1}{2}=
	\frac{\|B\|}{2\sqrt{(1-\|B\|^2)}}  .$$

\noindent Since $B\in \mathcal{B}(\mathcal{H})$, there exist a sequence $\{y_n \}$ in $\mathcal{H}$ with $\|y_n\|=1$ such that $\lim_{n \to \infty} \|By_n\| = \|B\|.$
Let $z_{n}=\frac{1}{\sqrt{\|By_n\|^2+k^2}}\left(\begin{array}{cc}
	By_n \\
	ky_n 
	\end{array}\right)$, where $k =  \frac{\|B\|}{\sqrt{1-2\|B\|^2}}$. Then
	\begin{eqnarray*}
	\lim_{n\rightarrow \infty} \{|\langle Tz_{n},z_{n} \rangle|^2+|\langle Tz_{n},Tz_{n} \rangle|^2\} ^{\frac{1}{2}}&=& \frac{\|B\|}{2\sqrt{1-\|B\|^2}}.
	\end{eqnarray*}
Therefore $dw(T)=\frac{\|B\|}{2\sqrt{(1-\|B\|^2)}}.$\\
Next we  consider the case  $\|B\| \geq \frac{1}{\sqrt{2}}$. Then $$\sup_{\theta \in [0,\frac{\pi}{2}]} \|B\|^2 \sin^2 \theta \left ( \cos^2 \theta +\|B\|^2 \sin^2 \theta \right )=\|B\|^4.$$
Therefore, $dw(T) \leq \|B\|^2. $ Now we show that there exist a sequence $\{z_n\}$ in $\mathcal{H} \oplus \mathcal{H}$ with $\|z_n\|=1$ such that 
$$ \lim_{n \to \infty}(|\langle Tz_n,z_n \rangle|^2+|\langle Tz_n,Tz_n \rangle|^2)^\frac{1}{2}=
	\|B\|^2 .$$
\noindent Since $B\in \mathcal{B}(\mathcal{H})$, there exist a sequence $\{y_n \}$ in $\mathcal{H}$ with $\|y_n\|=1$ such that $\lim_{n \to \infty} \|By_n\| = \|B\|.$	
 If we consider $z_n=\left(\begin{array}{cc}
	0 \\
	y_n 
	\end{array}\right)$, then $ \langle Tz_n,z_n \rangle = 0 $ and $ \langle Tz_n,Tz_n \rangle = \|By_n\|^2.$ Then $\lim_{n \to \infty}(|\langle Tz_n,z_n \rangle|^2+|\langle Tz_n,Tz_n \rangle|^2)^\frac{1}{2}=\|B\|^2.$ This completes the proof.
		\end{proof}

\begin{example}
Consider $B =\left(\begin{array}{cc}
	0 & 1\\
	0 & 1
	\end{array}\right)$ then $\|B\|=\sqrt{2}$. Then from Theorem \ref{32}, we have $dw\left(\begin{array}{cccc}
	0 & 0 & 0 & 1\\
	0 & 0 & 0 & 1\\
	0 & 0 & 0 & 0\\
	0 & 0 & 0 & 0
	\end{array}\right)= dw\left(\begin{array}{cc}
	0 & B\\
	0 & 0
	\end{array}\right) =\|B\|^2=2.$
\noindent Again if we consider $B =\left(\begin{array}{cc}
	0.3 & 0.4\\
	0 & 0.5
	\end{array}\right)$ then $\|B\|=0.671$. Then from Theorem \ref{32}, we have $dw\left(\begin{array}{cccc}
	0 & 0 & 0.3 & 0.4\\
	0 & 0 & 0 & 0.5\\
	0 & 0 & 0 & 0\\
	0 & 0 & 0 & 0
	\end{array}\right)= dw\left(\begin{array}{cc}
	0 & B\\
	0 & 0
	\end{array}\right)= \frac{\|B\|}{2\sqrt{1-\|B\|^2}}=0.452.$
	
\end{example}

Next, using Proposition \ref{eql}, we prove the following lower bound.

\begin{theorem}\label{upper29}
Let $A,B \in \mathcal{B}(\mathcal{H}) $. Then
\[\frac{1}{2} \Big ( \max \left \{ dw(A+B),dw(A-B) \right \} - \|A^*B+B^*A\| \Big ) \leq dw\left(\begin{array}{cc}
	0 & A\\
	B & 0
	\end{array}\right). \]
\end{theorem}

\begin{proof}
From Proposition \ref{eql}, we get
\begin{eqnarray*}
dw(A+B) &=& dw\left(\begin{array}{cc}
	0 & A+B\\
	A+B & 0
	\end{array}\right)  \\
	&\leq& dw\left(\begin{array}{cc}
	0 & A\\
	B & 0
	\end{array}\right)+dw\left(\begin{array}{cc}
	0 & B\\
	A & 0
	\end{array}\right)+w(A^*B+B^*A),~~\mbox{by Theorem} ~~\ref{lem24} \\
	&=& 2 dw\left(\begin{array}{cc}
	0 & A\\
	B & 0
	\end{array}\right)+\|A^*B+B^*A\|, ~~\mbox{by Lemma} ~~\ref{lem23}~~(b).
\end{eqnarray*}
So, \[ \frac{1}{2}\left (dw(A+B) -\|A^*B+B^*A\| \right ) \leq dw\left(\begin{array}{cc}
	0 & A\\
	B & 0
	\end{array}\right).\]
Replacing $B$ by $-B$, we get
\begin{eqnarray*}
\frac{1}{2}\left (dw(A-B) -\|A^*B+B^*A\| \right ) &\leq& dw\left(\begin{array}{cc}
	0 & A\\
	-B & 0
	\end{array}\right) \\
	&=& dw\left(\begin{array}{cc}
	0 & A\\
	B & 0
	\end{array}\right), ~~\mbox{by Lemma} ~~\ref{lem23}~~(a).
	\end{eqnarray*} 
	Combining the above two inequalities, we get the desired inequality.
\end{proof}

\begin{remark}
Here we would like to remark that there exist some operators for which the lower bound obtained in Theorem \ref{upper29} is sharper than the lower bound obtained in Theorem \ref{th-lower1}. As for example, if we consider $A=\left(\begin{array}{cc}
	1 & 0\\
	0 & 0
	\end{array}\right), B=\left(\begin{array}{cc}
	{\rm i} & 0\\
	0 & 0
	\end{array}\right)$ and $ T =  \left(\begin{array}{cc}
	0 & A\\
	B & 0
	\end{array}\right) $ then we see that Theorem \ref{upper29} gives $dw(T)\geq \frac{\sqrt{6}}{2}$, whereas Theorem \ref{th-lower1} gives $dw(T)\geq 1.$
\end{remark}

The following corollary immediately follows from Theorem \ref{upper29}.

\begin{cor}\label{cor30}
Let $A,B \in \mathcal{B}(\mathcal{H})$ be such that $A^*B+B^*A=0$. Then \[\frac{1}{2} \max \left \{ dw(A+B),dw(A-B) \right \} \leq dw\left(\begin{array}{cc}
	0 & A\\
	B & 0
	\end{array}\right). \]
	\end{cor}

Now, using Corollary \ref{prop25} and Theorem \ref{32}, we prove the following inequality.

\begin{theorem}\label{upper26}
Let $A,B \in \mathcal{B}(\mathcal{H})$, then
\[ dw\left(\begin{array}{cc}
	0 & A\\
	B & 0
	\end{array}\right) \leq \begin{cases}
		\frac{\|A\|}{2\sqrt{1-\|A\|^2}} + \frac{\|B\|}{2\sqrt{1-\|B\|^2}}, & \|A\| < \frac{1}{\sqrt{2}}, \|B\| < \frac{1}{\sqrt{2}} \\
		\frac{\|A\|}{2\sqrt{1-\|A\|^2}} + \|B\|^2, & \|A\| < \frac{1}{\sqrt{2}}, \|B\| \geq \frac{1}{\sqrt{2}} \\
		\|A\|^2 + \frac{\|B\|}{2\sqrt{1-\|B\|^2}}, &  \|A\| \geq \frac{1}{\sqrt{2}}, \|B\| < \frac{1}{\sqrt{2}}\\
		\|A\|^2 + \|B\|^2, & \|A\| \geq \frac{1}{\sqrt{2}}, \|B\| \geq \frac{1}{\sqrt{2}}.
		\end{cases}\]
\end{theorem}

\begin{proof}
Clearly $\left(\begin{array}{cc}
	0 & A\\
	0 & 0
	\end{array}\right)^*\left(\begin{array}{cc}
	0 & 0\\
	B & 0
	\end{array}\right)+\left(\begin{array}{cc}
	0 & 0\\
	B & 0
	\end{array}\right)^*\left(\begin{array}{cc}
	0 & A\\
	0 & 0
	\end{array}\right)=\left(\begin{array}{cc}
	0 & 0\\
	0 & 0
	\end{array}\right).$ \\
	 Therefore, using  Corollary \ref{prop25} and  Theorem \ref{32}, we get
	
	\begin{eqnarray*}
	dw\left(\begin{array}{cc}
	0 & A\\
	B & 0
	\end{array}\right) &\leq& dw\left(\begin{array}{cc}
	0 & A\\
	0 & 0
	\end{array}\right)+dw\left(\begin{array}{cc}
	0 & 0\\
	B & 0
	\end{array}\right)\\
	&=& dw\left(\begin{array}{cc}
	0 & A\\
	0 & 0
	\end{array}\right)+dw\left(\begin{array}{cc}
	0 & B\\
	0 & 0
	\end{array}\right),~~\mbox{By Lemma \ref{lem23}~~ (b)}\\
	&=& \begin{cases}
		\frac{\|A\|}{2\sqrt{1-\|A\|^2}} + \frac{\|B\|}{2\sqrt{1-\|B\|^2}}, & \|A\| < \frac{1}{\sqrt{2}}, \|B\| < \frac{1}{\sqrt{2}} \\
		\frac{\|A\|}{2\sqrt{1-\|A\|^2}} + \|B\|^2, & \|A\| < \frac{1}{\sqrt{2}}, \|B\| \geq \frac{1}{\sqrt{2}} \\
		\|A\|^2 + \frac{\|B\|}{2\sqrt{1-\|B\|^2}}, &  \|A\| \geq \frac{1}{\sqrt{2}}, \|B\| < \frac{1}{\sqrt{2}}\\
		\|A\|^2 + \|B\|^2, & \|A\| \geq \frac{1}{\sqrt{2}}, \|B\| \geq \frac{1}{\sqrt{2}}.
	\end{cases}
	\end{eqnarray*}
	\end{proof}

\begin{remark}
1. If $A=0$ or $B=0$ then the inequality in Theorem \ref{upper26} becomes equality.\\

2. The bounds for the Davis-Wielandt radius of operator matrices can be used to obtain bounds for the same of bounded linear operators.  
Consider $ T =  \left(\begin{array}{cccc}
	0 & 0 & 1 & 0\\
	0 & 0 & 0 & 1\\
	1 & 0 & 0 & 0\\
	0 & 1 & 0 & 0
\end{array}\right). $ Then \cite[Th. 2.1]{ZS} gives $dw(T) \leq \sqrt{6}$, whereas looking at $T$ as an operator matrix $ \left(\begin{array}{cc}
0 & A\\
B & 0
\end{array}\right),$ with $A=B=I,$ we get $dw(T) \leq 2.$ 
\end{remark}

Next, we need the following lemma, which can be found in \cite[pp. 75-76]{H}.

\begin{lemma} $($\cite[pp. 75-76]{H}$)$\label{lem5}
	Let $T\in \mathcal{B}(\mathcal{H}).$ Then for all $x\in \mathcal{H}$, we have
	\[|\langle Tx,x\rangle|\leq \langle |T|x,x\rangle^{1/2} \langle |T^*|x,x\rangle^{1/2}.\]
	\end{lemma}

Using Lemma \ref{lem5}, we obtain the following estimation for the Davis-Wielandt radius of operator matrices of the form $\left(\begin{array}{cc}
	0&A \\
	B&0
	\end{array}\right)$, where $A,B\in \mathcal{B}(\mathcal{H}).$	
	
\begin{theorem}\label{upper11}
Let $A,B\in \mathcal{B}(\mathcal{H}).$	Then 
\[dw^2\left(\begin{array}{cc}
	0&A \\
	B&0
	\end{array}\right)\leq \frac{1}{2}\max\left\{ \||B|^2+|A^*|^2+2|B|^4\|,  \| |A|^2+|B^*|^2+2|A|^4\|    \right\}.\]
	
\end{theorem}

\begin{proof}	
Let $T=\left(\begin{array}{cc}
	0&A \\
	B&0
	\end{array}\right)$ and $x=(x_1,x_2)\in \mathcal{H}\oplus \mathcal{H}$ with $\|x\|=1,$ i.e., $\|x_1\|^2+\|x_2\|^2=1.$ Now using Lemmas \ref{lem5} and \ref{lem7}, we get
	\begin{eqnarray*}
	&& |\langle Tx,x \rangle |^2+\|Tx\|^4\\
	&=& |\langle Tx,x \rangle |^2+ \langle T^*Tx,x \rangle^2\\
	&\leq & \langle |T|x,x \rangle \langle |T^*|x,x \rangle+\langle |T|^2x,x \rangle^2\\
	&\leq & \frac{1}{2}\left (\langle |T|x,x \rangle^2 +\langle |T^*|x,x \rangle^2\right)+\langle |T|^2x,x \rangle^2\\
	&\leq & \frac{1}{2}\left (\langle |T|^2x,x \rangle +\langle |T^*|^2x,x \rangle\right)+\langle |T|^4x,x \rangle\\
	&=& \frac{1}{2} \langle (|T|^2+|T^*|^2+2|T|^4)x,x\rangle\\
	&=& \frac{1}{2} \left \langle  \left(\begin{array}{cc}
	|B|^2+|A^*|^2+2|B|^4&0 \\
	0&|A|^2+|B^*|^2+2|A|^4
	\end{array}\right)x,x   \right \rangle \\
	&=& \frac{1}{2} \left \{\left \langle (|B|^2+|A^*|^2+2|B|^4)x_1,x_1\right \rangle+\left \langle (|A|^2+|B^*|^2+2|A|^4)x_2,x_2\right \rangle \right\}\\
	&\leq & \frac{1}{2} \max\left \{ \||B|^2+|A^*|^2+2|B|^4\|,\| |A|^2+|B^*|^2+2|A|^4\| \right \}.
	\end{eqnarray*}
	Therefore, taking supremum over all unit vectors in $\mathcal{H}\oplus \mathcal{H}$, we get the required inequality.	
\end{proof}

\begin{remark}
1. Consider $ T =  \left(\begin{array}{cccc}
0 & 0 & 1 & 0\\
0 & 0 & 0 & 0\\
0 & 0 & 0 & 0\\
0 & {\rm i} & 0 & 0
\end{array}\right) = \left(\begin{array}{cc}
0 & A\\
B & 0
\end{array}\right), $ where  $A=\left(\begin{array}{cc}
1 & 0\\
0 & 0
\end{array}\right)$ and $B=\left(\begin{array}{cc}
0 & 0\\
0 & {\rm i}
\end{array}\right).$ Then Theorem \ref{upper11} gives $dw^2(T) \leq
 \frac{3}{2}$, whereas the inequality in Theorem  \ref{upper15} (i) gives  $dw^2(T) \leq 2.$
 This shows that estimation of bounds for  the Davis-Wielandt radius of a bounded linear operator as an operator matrix is a better one. \\

2. Here we would like to note that the inequalities in Theorem \ref{upper26} and Theorem \ref{upper11} are not comparable, in general and this incomparability follows from by considering two matrices  $\left(\begin{array}{cc}
0 & 1\\
1 & 0
\end{array}\right)$ and $\left(\begin{array}{cc}
0 & 0\\
1 & 0
\end{array}\right)$.
\end{remark}

Finally, we compute the exact value for the Davis-Wielandt radius of $\left(\begin{array}{cc}
0 & A\\
B & 0
\end{array}\right)$, under  certain conditions.

\begin{theorem}\label{off}
	Let $A,B $ be two non-zero bounded linear operators on $\mathcal{H}$. 
	Then an upper bound for the Davis-Wielandt radius of  $T=\left(\begin{array}{cc}
	0 & A\\
	B & 0
	\end{array}\right)$ is given by 
	\begin{equation*}
dw^2(T) \leq\begin{cases}
		\frac{\left(\|A\|+\|B\| \right )^2+4\|A\|^2\|B\|^2}{4 \left\{1-(\|A\|-\|B\|)^2 \right\}},    & 
		\frac{-1}{2 \|B\|^2} < \frac{\|A\|-\|B\|}{\|A\| + \|B\|} < \frac{1}{2\|A\|^2}\\
	\max \left\{ \|A\|^4, \|B\|^4 \right \},             & \mbox{otherwise}.
	\end{cases}
	\end{equation*}
Moreover, the upper bound is attained if  there exists $~u_0,v_0\in \mathcal{H}$ such that $Au_0=\|A\|u_0~,~Bv_0=\|B\|v_0$ and $u_0=\lambda v_0$ for some $\lambda \in \mathbb{R}.$
\end{theorem}

\begin{proof} We note that if  $\frac{-1}{2 \|B\|^2} < \frac{\|A\|-\|B\|}{\|A\| + \|B\|} < \frac{1}{2\|A\|^2} $ then  $ (\|A\|-\|B\|)^2 -1 < 0.$ Let   
	\begin{eqnarray*}
	 f(\|A\|, \|B\|)   =    
	\begin{cases}
	\frac{\left(\|A\|+\|B\| \right )^2+4\|A\|^2\|B\|^2}{4 \left\{1-(\|A\|-\|B\|)^2 \right\}},    & 
	\frac{-1}{2 \|B\|^2} < \frac{\|A\|-\|B\|}{\|A\| + \|B\|} < \frac{1}{2\|A\|^2}\\
	\max \left\{ \|A\|^4, \|B\|^4 \right \} ,            & \mbox{otherwise}.
	\end{cases}
	\end{eqnarray*}

	We first show that $	f(\|A\|, \|B\|)$ is an upper bound for $ dw^2(T). $
	Let us choose $z=\left(\begin{array}{c}
	u\\
	v
	\end{array}\right) \in \mathcal{H}\oplus \mathcal{H}$ such that $\|z\|^2=\|u\|^2 +\|v\|^2=1.$ Consider $\|u\|=\cos\theta$ and $\|v\|=\sin\theta$. Then, we have  
	\begin{eqnarray*}
		|\langle Tz,z \rangle|^2+\|Tz\|^4 &=&|\langle Av,u \rangle + \langle Bu,v \rangle|^2 + (\|Av\|^2+\|Bu\|^2)^2 \\
		&\leq&  (\|A\|\|v\|\|u\| + \|B\|\|u\|\|v\|)^2 + (\|A\|^2\|v\|^2+\|B\|^2\|u\|^2)^2 \\
		&=& (\|A\|+\|B\|)^2 \|u\|^2 \|v\|^2+(\|A\|^2\|v\|^2 +\|B\|^2\|u\|^2)^2 \\
		&=& (\|A\|+\|B\|)^2 \cos^2  \theta \sin^2 \theta +(\|A\|^2\sin^2 \theta+\|B\|^2\cos^2 \theta)^2 \\
		&=& \frac{1}{8} (\alpha \cos 4\theta + \beta\cos 2\theta + \gamma),
	\end{eqnarray*}
	where 
	$\alpha=(\|B\|+\|A\|)^2 \Big((\|B\|-\|A\|)^2-1\Big)$,
	$\beta=4(\|B\|^4-\|A\|^4)$,
	$\gamma=(\|B\|^2-\|A\|^2)^2+2(\|B\|^2+\|A\|^2)^2+(\|B\|+\|A\|)^2.$ Considering $ g(\theta)=  \frac{1}{8} (\alpha \cos 4\theta + \beta\cos 2\theta + \gamma) $ and using elementary calculus we can  
	show  that $f(\|A\|,\|B\|)$ is an upper bound for $dw^2(T)$.
If, in addition,  there exists $~u_0,v_0\in \mathcal{H}$ such that $Au_0=\|A\|u_0~,~Bv_0=\|B\|v_0$ and $u_0=\lambda v_0$ for some $\lambda \in \mathbb{R}$, then  it is easy to see that the upper bound for $dw^2(T)$ is attained at $z_0=\frac{1}{\sqrt{\|u_0\|^2+\|v_0\|^2}}\left(\begin{array}{c}
	v_0\\
	u_0
	\end{array}\right)$.
	Therefore, $ dw^2(T) = f(\|A\|,\|B\|).$ 
\end{proof}

\begin{remark}
1. If we consider $A=B$ then from Theorem \ref{upper11}  we get  $$dw^2\left(\begin{array}{cc}
	0&A \\
	A&0
	\end{array}\right)\leq \frac{1}{2}\left\{ \||A|^2+|A^*|^2+2|A|^4\| \right\}  $$ whereas from  Theorem \ref{off} we get $$ dw^2\left(\begin{array}{cc}
	0&A \\
	A&0
	\end{array}\right)\leq \|A\|^2+\|A\|^4.$$ 
	Clearly $  \frac{1}{2}\left\{ \||A|^2+|A^*|^2+2|A|^4\| \right\}   \leq  \|A\|^2+\|A\|^4 $ so that Theorem \ref{upper11} always gives a better bound for $ A=B.$ On the other hand, if we consider $A=2I, B=I,$ then Theorem \ref{upper11} gives $dw^2\left(\begin{array}{cc}
	0&A \\
	B&0
	\end{array}\right)\leq \frac{37}{2} $, whereas Theorem \ref{off} gives $dw^2\left(\begin{array}{cc}
	0&A \\
	B&0
	\end{array}\right)\leq  16$. This shows that the bounds obtained in Theorem \ref{upper11} and Theorem \ref{off}  are not comparable in general.

\end{remark}

 The final results of this section is  to  obtain upper bounds for the Davis-Wielandt radius of operator matrices of the form $\left(\begin{array}{cc}
	A & B\\
	0 & C
	\end{array}\right)$, where $A,B,C \in \mathcal{B}(\mathcal{H}).$

\begin{theorem}\label{35}
Let $A,B,C \in \mathcal{B}(\mathcal{H}).$ Then
$$dw^2\left(\begin{array}{cc}
	A & B\\
	0 & C
	\end{array}\right) \leq \frac{9}{4}\max \Big \{\|A\|^2,\|B\|^2,\|C\|^2 \Big \}+\frac{14+6\sqrt{5}}{4}\max \Big\{\|A\|^4,\|B\|^4,\|C\|^4 \Big\}.$$
\end{theorem}

\begin{proof}
Let $T=\left(\begin{array}{cc}
	A & B\\
	0 & C
	\end{array}\right)$. Let $z=\left(\begin{array}{cc}
	x \\
	y 
	\end{array}\right) \in \mathcal{H}\oplus \mathcal{H} $ be such that $\|z\|=1$.	Then, we have
\begin{eqnarray*}
dw^2(T) &=&\sup_{\|z\|=1} \{|\langle Tz,z \rangle|^2+\|Tz\|^4\}\\
&=& \sup_{\|x\|^2+\|y\|^2=1}\{|\langle Ax+By,x\rangle+\langle Cy,y\rangle|^2+(\|Ax+By\|^2+\|Cy\|^2)^2\}\\
&\leq& \sup_{\|x\|^2+\|y\|^2=1} \{(\|Ax+By\|\|x\|+\|Cy\|\|y\|)^2 + ((\|Ax\|+\|By\|)^2+\|Cy\|^2)^2\}\\
&\leq& \sup_{\|x\|^2+\|y\|^2=1} \Big \{(\|A\|\|x\|^2+\|B\|\|x\|\|y\|+\|C\|\|y\|^2)^2\\
&& + ((\|A\|\|x\|+\|B\|\|y\|)^2+\|C\|^2\|y\|^2)^2 \Big \}\\
&\leq& \sup_{\|x\|^2+\|y\|^2=1}  \Big\{ \max \{\|A\|^2,\|B\|^2,\|C\|^2 \}(\|x\|^2+\|y\|^2+\|x\|\|y\|)^2 \\
&& +\max \{\|A\|^4,\|B\|^4,\|C\|^4 \} (\|x\|^2+\|y\|^2+2\|x\|\|y\|+\|y\|^2)^2 \Big \}\\
&=& \sup_{\theta \in [0,\frac{\pi}{2}]} \Big\{ \max \{\|A\|^2,\|B\|^2,\|C\|^2 \}(1+\cos \theta\sin \theta)^2 \\
&& +\max \{\|A\|^4,\|B\|^4,\|C\|^4 \} (1+2\cos \theta\sin \theta+\sin^2\theta)^2 \Big\}\\
&\leq& \frac{9}{4}\max \{\|A\|^2,\|B\|^2,\|C\|^2 \}+\frac{14+6\sqrt{5}}{4}\max \{\|A\|^4,\|B\|^4,\|C\|^4 \}. 
\end{eqnarray*}

\end{proof}

\begin{remark}
Here we would like to note that there exist some operators for which the bound obtained in Theorem \ref{35} is sharper than the bounds obtained in Theorem \ref{th-upper2} and Theorem \ref{upper15} (i). As for example, we consider $T=\left(\begin{array}{cc}
	1 & 1\\
	0 & 1
	\end{array}\right)$. If looking at $T$ as an operator matrix $ \left(\begin{array}{cc}
A & B\\
0 & C
\end{array}\right),$ with $A=B=C=I,$ then  Theorem \ref{35} gives $dw^2(T) \leq 9.104$, whereas Corollary \ref{upper8} (ii), Theorem \ref{upper15} (i) and Theorem \ref{upper15} (ii)  give $dw^2(T) \leq 9.272$, $dw^2(T) \leq 9.472$ and $dw^2(T) \leq 9.162,$ respectively.
\end{remark}

\begin{theorem}\label{34}
Let $A,B,C \in \mathcal{B}(\mathcal{H}).$ Then
$$dw\left(\begin{array}{cc}
	A & B\\
	0 & C
	\end{array}\right) \leq \max \Big \{dw(A),dw(C) \Big \}+ \|A^*B\|+\begin{cases}
		\frac{\|B\|}{2\sqrt{1-\|B\|^2}}, & \|B\| < \frac{1}{\sqrt{2}} \\
		\|B\|^2, & \|B\| \geq \frac{1}{\sqrt{2}}.
		\end{cases}$$
\end{theorem}

\begin{proof}
Using Lemma \ref{lem24}, Lemma  \ref{33} and Theorem \ref{32}, we get
\begin{eqnarray*}
dw\left(\begin{array}{cc}
	A & B\\
	0 & C
	\end{array}\right) &=& dw\left(\left(\begin{array}{cc}
	A & 0\\
	0 & C
	\end{array}\right)+ \left(\begin{array}{cc}
	0 & B\\
	0 & 0
	\end{array}\right)\right)\\
	&\leq& dw\left(\begin{array}{cc}
	A & 0\\
	0 & C
	\end{array}\right)+ dw\left(\begin{array}{cc}
	0 & B\\
	0 & 0
	\end{array}\right) \\
	&& + w\left (\left(\begin{array}{cc}
	A & 0\\
	0 & C
	\end{array}\right)^* \left(\begin{array}{cc}
	0 & B\\
	0 & 0
	\end{array}\right)+ \left(\begin{array}{cc}
	0 & B\\
	0 & 0
	\end{array}\right)^*\left(\begin{array}{cc}
	A & 0\\
	0 & C
	\end{array}\right) \right)\\
	&=& \max \{dw(A),dw(C) \}+dw\left(\begin{array}{cc}
	0 & B\\
	0 & 0
	\end{array}\right)+w\left(\begin{array}{cc}
	0 & A^*B\\
	B^*A & 0
	\end{array}\right)\\
	&=& \max \{dw(A),dw(C) \} + \|A^*B\| +\begin{cases}
		\frac{\|B\|}{2\sqrt{1-\|B\|^2}}, & \|B\| < \frac{1}{\sqrt{2}} \\
		\|B\|^2, & \|B\| \geq \frac{1}{\sqrt{2}}.
		\end{cases}
\end{eqnarray*}
This completes the proof.
\end{proof}

\begin{remark}
We note that if we consider $B=0$ in Theorem \ref{34} then the inequality becomes equality.
\end{remark}

\begin{remark}
The inequalities in Theorem \ref{35} and Theorem \ref{34} are not comparable, in general. If we consider $B=0$ then Theorem \ref{34} gives better bound than that in Theorem \ref{35}. Again if we consider $T=\left(\begin{array}{cc}
	1 & 1\\
	0 & 0
	\end{array}\right)$ then Theorem \ref{35} gives $dw(T)\leq \sqrt{\frac{23+6\sqrt{5}}{4}} =3.017$, whereas Theorem \ref{34} gives $dw(T)\leq 2+\sqrt{2}= 3.414$.
\end{remark}

\textbf{Acknowledgement.}  We would like to thank the referee for his/her valuable suggestions. The comparaibilty of bounds in section 2 was suggested by the referee, he also suggested to compute the bounds for some of $ 2 \times 2$ operator matrices. The present form of the paper is due to him/her.
\bibliographystyle{amsplain}

\begin{thebibliography}{99}





%\bibitem{BBP1} S. Bag, P. Bhunia  and K. Paul, Bounds of numerical radius of bounded linear operator using $t$-Aluthge transform, Math. Inequal. Appl. (2020) Accepted. \url{arXiv:1904.12096v3 [math.FA]}. 

\bibitem{BBP} P. Bhunia, S. Bag and K. Paul, Numerical radius inequalities and its applications in estimation of zeros of polynomials, Linear Algebra Appl. 573 (2019) 166-177.

%\bibitem{BBP2} P. Bhunia, S. Bag and K. Paul, Numerical radius inequalities of operator matrices with applications, Linear Multilinear Algebra, (2019), \url{https://doi.org/10.1080/03081087.2019.1634673}.

\bibitem {B} M.L. Buzano, Generalizzatione della diseguaglianza di Cauchy-Schwarz, Rend. Sem. Mat. Univ. e Politech. Trimo 31 (1971/73) 405-409.

\bibitem{D} C. Davis, The shell of a Hilbert-space operator, Acta Sci. Math., (Szeged) 29 (1968) 69-86.

\bibitem{Dra} S.S. Dragomir, Reverses of Schwarz inequality in inner product spaces with applications, Math. Nachr. 288(7) (2015) 730-742.

%\bibitem{Dra2} S.S. Dragomir, Operator refinements of Schwarz inequality in inner product spaces, Linear Multilinear Algebra, (2018), \url{ https://doi.org/10.1080/03081087.2018.1472204}.

\bibitem{DB} J. Duncan and F.F. Bonsal, Numerical Ranges II, Cambridge University Press, 2013, \url{https://doi.org/10.1017/CBO9780511662515}.

\bibitem {GR} K.E. Gustafson and D.K.M. Rao, Numerical Range, The field of values of linear operators and matrices, Springer, New York, 1997.

\bibitem{H} P.R. Halmos, A Hilbert space problem book, 2nd ed., Springer, New York, 1982.

\bibitem{LP} C.-K. Li and Y.-T. Poon, Spectrum, numerical range and Davis-Wielandt Shell of a normal operator, Glasgow Math. J. 51 (2009) 91-100.

\bibitem{LPS}  C.-K. Li, Y.-T. Poon and N.S. Sze, Davis-Wielandt shells of operators, Oper. Matrices,  2(3) (2008) 341-355, \url{https://dx.doi.org/10.7153/oam-02-20}.

\bibitem{LSZ} B. Lins, I.M. Spitkovsky, S. Zhong, The normalized numerical range and the Davis-Wielandt shell, Linear Algebra Appl. 546 (2018) 187-209.

\bibitem {PB} K. Paul and S. Bag, On the numerical radius of a matrix and estimation of bounds for zeros of a polynomial, Int. J. Math. Math. Sci. 2012 (2012) Article Id 129132, \url{https://doi.org/10.1155/2012/129132}.

\bibitem{K} F. Kittaneh, Notes on some inequalities for Hilbert space operators, Publ. RIMS Kyoto Univ. 24 (1988) 283-293.

\bibitem {SMY} M. Sattari, M.S. Moslehian and T. Yamazaki, Some generalized numerical radius inequalities for Hilbert space operators, Linear Algebra Appl. 470 (2015) 216-227.

\bibitem{W} H. Wielandt, On eigenvalues of sums of normal matrices, Pac. J. Math. 5 (1955) 633-638. 
 
\bibitem{ZMCN}  A. Zamani, M.S. Moslehian, M.-T. Chien and H. Nakazato, Norm-parallelism and the Davis-Wielandt radius of Hilbert space operators, Linear
Multilinear Algebra, 67(11), (2019) 2147-2158.

\bibitem{ZS} A. Zamani and K. Shebrawi, Some Upper Bounds for the Davis-Wielandt Radius of Hilbert Space Operators, Mediterr. J. Math., (2019), \url{https://doi.org/10.1007/s00009-019-1458-z}.




\end{thebibliography}

\end{document}